\documentclass[leqno]{amsart}
\usepackage[utf8]{inputenc}

\usepackage{amsmath}
\usepackage{amsthm}
\usepackage{enumerate}
\usepackage{mathtools}
\usepackage{stmaryrd}
 
\usepackage[colorlinks=false]{hyperref}

\usepackage{color}

\makeatletter
\g@addto@macro\bfseries{\boldmath}
\makeatother

\usepackage{xy}
\xyoption{all}

\theoremstyle{plain}
    \newtheorem{theorem}[equation]{Theorem}
    \newtheorem{lemma}[equation]{Lemma}
    \newtheorem{corollary}[equation]{Corollary}
    \newtheorem{proposition}[equation]{Proposition}

    \newtheorem*{theorem*}{Theorem}
    \newtheorem*{proposition*}{Proposition}
    \newtheorem*{corollary*}{Corollary}
    \newtheorem*{lemma*}{Lemma}
    \newtheorem*{conjecture*}{Conjecture}
    \newtheorem{definition-theorem}[equation]{Definition/Theorem}
    \newtheorem{definition-lemma}[equation]{Definition/Lemma}
\theoremstyle{definition}
    \newtheorem{definition}[equation]{Definition}

    \newtheorem{remark}[equation]{Remark}
    
    \newtheorem{remarks}[equation]{Remarks}

    \newcommand{\C}{\mathbb{C}}
    \newcommand{\N}{\mathbb{N}}
    \newcommand{\Z}{\mathbb{Z}}

   	\renewcommand{\phi}{\varphi}
	\let\epsilon\varepsilon

    \newcommand{\Bounded}{\operatorname{B}}
    \newcommand{\CB}{\operatorname{CB}}

    \newcommand{\cb}{\mathrm{cb}}
    
   \newcommand{\functor}{\mathsf}

\newcommand{\into}{\hookrightarrow}

\newcommand{\id}{\mathrm{id}}

  \renewcommand{\prod}{\bigsqcap}

 \newcommand{\flip}{\tau}
 
 \newcommand{\h}{\mathrm{h}}

 	\DeclareMathOperator{\opmod}{OM}
	\DeclareMathOperator{\CMod}{CM}
	\DeclareMathOperator{\CComod}{CC}
	\DeclareMathOperator{\Des}{Des}
	\newcommand{\module}{}

\title[Commutativity of the Haagerup tensor product]{Commutativity of the Haagerup tensor product and base change for operator modules}
\author{Tyrone Crisp }
\date{December 1, 2020}
\address{Department of Mathematics \& Statistics, University of Maine.
5752 Neville Hall, Room 333.
Orono, ME 04469 USA}
\email{tyrone.crisp@maine.edu}



\begin{document}

\begin{abstract}
By computing the completely bounded norm of the flip map on the Haagerup tensor product $C_0 Y_1\otimes_{C_0 X} C_0 Y_2$ associated to a pair of continuous mappings of locally compact Hausdorff spaces $Y_1\rightarrow X\leftarrow Y_2$, we establish a simple characterisation of the Beck-Chevalley condition for base change of operator modules over commutative $C^*$-algebras, and  a descent theorem for continuous fields of Hilbert spaces.
\end{abstract}

\maketitle

It is known (cf.~\cite{Blecher-donotpreserve}) that the canonical map $B_1\otimes^{\h} B_2\to B_1\otimes^{\min} B_2$ from the Haagerup tensor product of two $C^*$-algebras to their minimal $C^*$-algebra tensor product is a completely bounded isomorphism if and only if one of the $B_i$  is finite dimensional. In this note we shall establish a relative version of this result in the commutative case, and draw some geometric conclusions. 


Let $\mu_1:Y_1\to X$ and $\mu_2:Y_2\to X$ be a pair of continuous maps of locally compact Hausdorff topological spaces. We wish to compare the $C^*$-algebra $C_0(Y_1\times_X Y_2)$ of continuous, complex-valued, vanishing-at-infinity functions on the fibred product 
\[
Y_1\times_X Y_2 = \{ (y_1,y_2)\in Y_1\times Y_2 \ |\ \mu_1 y_1=\mu_2y_2\}
\]
with the relative Haagerup tensor product $C_0 Y_1\otimes_{C_0 X}^{\h} C_0 Y_2$. See \cite{BLM}, or below, for the definition of this tensor product; for now it will be important to note only that both $C_0(Y_1\times_X Y_2)$ and $C_0 Y_1\otimes^{\h}_{C_0X} C_0 Y_2$ are \emph{operator spaces}, so that there are canonical norms on the matrix spaces $M_n(C_0(Y_1\times_X Y_1))$ and $M_n( C_0 Y_1\otimes^{\h}_{C_0 X} C_0 Y_2)$.

The  map sending each elementary tensor $b_1\otimes b_2 \in C_0 Y_1\otimes C_0 Y_2$ to the function $(y_1,y_2)\mapsto b_1(y_1)b_2(y_2)$ extends to a map
\begin{equation}\label{eq:comparison}\tag{$\star$}
C_0 Y_1\otimes^{\h}_{C_0 X} C_0 Y_2 \to C_0(Y_1\times_X Y_2),
\end{equation}
and the question that we shall answer is: under what circumstances is \eqref{eq:comparison} a completely bounded isomorphism? (We recall that a linear map $t:E\to F$ between operator spaces is \emph{completely bounded} if the induced maps on matrix spaces $M_n(E)\to M_n(F)$ are  uniformly bounded; and $t$ is a \emph{completely bounded isomorphism} if it is a linear isomorphism and both $t$ and $t^{-1}$ are completely bounded.)

The noncommutativity of the Haagerup tensor product presents an obvious obstruction to \eqref{eq:comparison} being a completely bounded isomorphism: in general the flip map $b_1\otimes b_2\mapsto b_2\otimes b_1$ is not completely bounded with respect to the Haagerup norm, whereas this map is completely isometric for the $C^*$-algebra norm. We shall prove that this noncommutativity is the only obstruction, and characterise the vanishing of the obstruction in terms of a finiteness condition on the maps $\mu_i$:

\begin{theorem}\label{thm:main}
The following are equivalent for continuous maps $\mu_1:Y_1\to X$ and $\mu_2:Y_2\to X$:
\begin{enumerate}[\rm(a)]
\item The canonical map \eqref{eq:comparison} is a completely bounded isomorphism.
\item The flip map $C_0 Y_1\otimes_{C_0 X}^{\h} C_0 Y_2 \to C_0 Y_2 \otimes_{C_0 X}^{\h} C_0 Y_1$ is completely bounded.
\item $\displaystyle\max_{x\in X}\min_{i=1,2}\# \mu_i^{-1}x<\infty$.
\end{enumerate}
\end{theorem}
We in fact obtain a quantitative version of the equivalences by relating the completely bounded norms of the flip and of the inverse of \eqref{eq:comparison} to the quantity appearing in (c); see Lemma \ref{lem:pi1pi2} and Proposition \ref{prop:flip}.

The study of the map \eqref{eq:comparison} is quite natural  from the point of view of abstract harmonic analysis: $C_0 Y_1\otimes^{\h}_{C_0 X} C_0 Y_2$ is a commutative Banach algebra, isomorphic (by the Grothendieck inequality, see e.g.~\cite{Pisier-Grothendieck}) to the relative projective tensor product $C_0 Y_1\widehat{\otimes}_{C_0 X} C_0 Y_2$, and the map \eqref{eq:comparison} is the Gelfand transform (cf.~\cite[Theorem 1]{Gelbaum-TPOBA}).

Our motivation for studying \eqref{eq:comparison} comes from quite a different direction, as we shall now explain. 

\subsection*{Base change of operator modules}

To each $C^*$-algebra $A$ is associated a category $\opmod(A)$ of (nondegenerate, right) \emph{operator modules} over $A$, with completely bounded $A$-linear maps as morphisms. Concretely, an operator $A$-module is a norm-closed linear subspace $E\subseteq \Bounded(H)$, where $H$ is a Hilbert space on which $A$ is represented as bounded operators, satisfying $EA=E$; see \cite{BLM} for details. 

The most important property of operator modules for our purpose here is that they can be both pushed forward and pulled back along homomorphisms of $C^*$-algebras. In particular, to each continuous map of locally compact Hausdorff spaces $\mu:Y\to X$ there is associated a pair of functors 
\[
\mu_*:\opmod(C_0 Y)\to \opmod(C_0 X)\quad \text{and}\quad \mu^*:\opmod(C_0 X) \to \opmod(C_0 Y)
\]
given by Haagerup tensor product with $C_0 Y$, viewed either as a $C_0 Y$-$C_0 X$ bimodule (in the case of $\mu_*$) or as a $C_0 X$-$C_0 Y$-bimodule (in the case of $\mu^*$). If the map $\mu$ is proper then the functor $\mu^*$ is left-adjoint to $\mu_*$, meaning that there are natural isomorphisms between the spaces of completely bounded module maps
\[
\CB_{C_0X}(F,\mu_* E) \cong \CB_{C_0 Y}(\mu^* F, E)
\]
for all $F\in \opmod(C_0X)$ and $E\in \opmod(C_0 Y)$; see \cite[Theorem 2.17]{CH-Kadison}.

Now suppose that we have a pair of {proper}, continuous maps $\mu_1:Y_1\to X$ and $\mu_2:Y_2\to X$. 
The fibred product $Y_1\times_X Y_2$ fits into a commuting diagram
\begin{equation*}\label{eq:pullback-diagram}
\xymatrix@R=10pt{
& Y_1\times_X Y_2 \ar[dl]_-{\pi_1} \ar[dr]^-{\pi_2} & \\
Y_1 \ar[dr]_-{\mu_1} & & Y_2 \ar[dl]^-{\mu_2}\\
&X&
}
\end{equation*}
where $\pi_i(y_1,y_2)\coloneqq y_i$, and we consider the pair of functors
\[
\xymatrix@C=60pt{
\opmod(C_0 Y_1) \ar@<.5ex>[r]^-{\pi_{2*} \pi_1^*}   \ar@<-.5ex>[r]_-{ \mu_2^* \mu_{1*}}& \opmod(C_0 Y_2).
}
\]
The adjunction between pullbacks and pushforwards yields, as observed in \cite{BenRou}, a canonical natural transformation $\mu_2^*\mu_{1*} \to \pi_{2*} \pi_1^*$ called the \emph{base-change} map associated to $\mu_1$ and $\mu_2$, and the maps $\mu_1$ and $\mu_2$ are said to satisfy the \emph{Beck-Chevalley condition} (for operator modules) if the base-change map is an isomorphism.  This base-change map is \eqref{eq:comparison} by another name: the functors $\mu_2^* \mu_{1*}$ and $\pi_{2*}\pi_1^*$ are given by Haagerup tensor product with the operator bimodules $C_0 Y_1\otimes^{\h}_{C_0 X} C_0 Y_2$ and $C_0 (Y_1\times_X Y_2)$, respectively, and the base-change map is the natural transformation of functors induced by the bimodule map \eqref{eq:comparison}. Theorem \ref{thm:main} thus implies:

\begin{corollary}\label{cor:BC}
The maps $\mu_1$ and $\mu_2$ satisfy the Beck-Chevalley condition for operator modules if and only if $\displaystyle\max_{x\in X}\min_{i=1,2}\# \mu_i^{-1}x<\infty$.\hfill\qed
\end{corollary}

\subsection*{Descent of continuous fields of Hilbert spaces} In \cite{BenRou} B\'enabou and Roubaud showed, in a very general setting, that the Beck-Chevalley condition leads to an identification between two approaches to the \emph{descent problem}.  Let us explain what this means in the present context. 

Let $\mu:Y\to X$ be a continuous, proper, surjective map of locally compact Hausdorff spaces, and consider the problem of characterising the image of the pullback functor $\mu^*: \opmod(C_0 X)\to \opmod(C_0 Y)$. In \cite{Crisp-descent} we gave a solution to this problem: an operator module $\module{E}$ over $C_0 Y$ has the form $\mu^*\module{F}$ for some operator module $\module{F}$ over $C_0 X$ if and only if there exists a completely isometric $C_0 Y$-linear splitting $\module{E}\to \mu^* \mu_* \module{E}$ of the multiplication map $\mu^* \mu_*\module{E}\to \module{E}$, satisfying a coassociativity condition; see \cite[Theorem 4.1]{Crisp-descent}. If  $\mu$ satisfies the Beck-Chevalley condition---which, by Corollary \ref{cor:BC}, is equivalent to the condition that the fibres of $\mu$ should have uniformly bounded finite cardinality---then we have natural isomorphisms of spaces of completely bounded module maps
\begin{equation}\label{eq:BenRou}
\CB_{C_0 Y}(\module{E}, \mu^*\mu_* \module E) \xrightarrow{\cong}\CB_{C_0 Y}(\module{E}, \pi_{2*}\pi_1^*\module{E}) \xrightarrow{\cong} \CB_{C_0(Y\times_X Y)}( \pi_2^*\module{E}, \pi_1^*\module E)
\end{equation}
coming from base change and from the adjunction between pullback and pushforward. In this way the existence of a splitting $\module{E}\to \mu^*\mu_* \module{E}$ as studied in \cite{Crisp-descent} can be reformulated as the existence of a completely bounded, $C_0(Y\times_X Y)$-linear map $\pi_2^*\module{E}\to \pi_1^*\module{E}$ satisfying certain properties, stated explicitly below.

The advantage of this reformulation becomes apparent when one's primary interest is not in the category $\opmod$ of operator modules and completely bounded maps, but in the subcategory $\CMod$  of \emph{Hilbert $C^*$-modules}  and adjointable maps. Let us briefly recall that a Hilbert $C^*$-module over a $C^*$-algebra $A$ is a right $A$-module  equipped with an $A$-valued inner product satisfying  analogues of the axioms of a ($\C$-valued) Hilbert-space inner product; while an adjointable map of Hilbert $C^*$-modules $t:F\to E$ is one for which there exists a map $t^*:E\to F$ satisfying $\langle tf|e\rangle=\langle f|t^*e\rangle$. See \cite{Lance} or \cite{BLM} for details. A Hilbert $C^*$-module over a commutative $C^*$-algebra $C_0X$ is (as shown in \cite{Takahashi}) the same thing as a continuous field of Hilbert spaces over $X$. By contrast, the notion of an operator module over $C_0X$ is more general, encompassing for instance also  {measurable} fields of Hilbert spaces over $X$, and much more besides.

A theorem of Blecher \cite[Theorem 4.3]{Blecher-newapproach} implies that the pullback of a Hilbert $C^*$-module along a map of spaces (defined, as above, by a Haagerup tensor product) is again a Hilbert $C^*$-module. But the same is not generally true of pushforwards, and this means that the characterisation of the image of the pullback functor $\mu^*:\CMod(C_0 X)\to \CMod(C_0 Y)$ in terms of splittings $\module{E}\to \mu^*\mu_* \module{E}$ given in \cite[Theorem 5.6]{Crisp-descent} necessarily involves operator modules that are not Hilbert $C^*$-modules. Using \eqref{eq:BenRou} to replace splittings $\module{E}\to \mu^*\mu_* E$ by maps $\pi_2^*\module{E}\to \pi_1^*\module{E}$ allows us to reformulate this characterisation entirely within the category of Hilbert $C^*$-modules. To state the result (Corollary \ref{cor:descent}, below) we shall need some terminology, adapted from \cite{Grothendieck}.

\begin{definition}
A \emph{descent datum} (in Hilbert $C^*$-modules) for a continuous proper map $\mu:Y\to X$ is a pair $(\module{E},\phi)$ where $\module{E}$ is a Hilbert $C^*$-module over $C_0 Y$, and where $\phi: \pi_2^* \module{E} \to \pi_1^* \module{E}$ is a unitary isomorphism of Hilbert $C^*$-modules over $C_0(Y\times_X Y)$ satisfying $\pi_{12}^*(\phi)  \pi_{23}^*(\phi) = \pi_{13}^*(\phi)$. Here $\pi_{ij}$ is the map $(y_1,y_2,y_3)\mapsto (y_i,y_j)$ from $Y\times_X Y\times_X Y$ to $Y\times_X Y$, and we are tacitly invoking  isomorphisms such as $\pi_{23}^*\pi_1^*\cong (\pi_1 \pi_{23})^* = (\pi_2 \pi_{12})^* \cong \pi_{12}^*\pi_2^*$, etc.

A morphism of descent data $(\module{E}_1,\phi_1)\to (\module{E}_2,\phi_2)$ is an adjointable mapping of Hilbert $C^*$-modules $t:\module{E}_1\to \module{E}_2$ satisfying $\pi_1^*( t)  \phi_1 = \phi_2 \pi_2^* (t)$.
Equipped with these morphisms, the descent data for $\mu$ form a $C^*$-category which we denote by $\Des\mu$. (We refer to \cite{GLR} for the notion of a $C^*$-category.)

For each Hilbert $C^*$-module $\module{F}\in \CMod(C_0 X)$ we let $\phi_{\module{F}}:\pi_2^*\mu^*\module{F}\to \pi_1^*\mu^* \module{F}$ be the canonical unitary isomorphism coming from the equality $\mu\pi_2=\mu\pi_1$.  The   assignment $\module{F}\mapsto (\mu^*\module{F},\phi_{\module{F}})$ extends to a $*$-functor $\functor{D}:\CMod(C_0 X)\to \Des\mu$, defined on morphisms by $\functor{D}(t)\coloneqq \mu^* (t)$. 
\end{definition}

\begin{corollary}\label{cor:descent}
Let $\mu:Y\to X$ be a proper, surjective mapping of locally compact Hausdorff spaces satisfying $\max_{x\in X} \#\mu^{-1}x<\infty$. The functor $\functor{D}$  is a unitary equivalence of $C^*$-categories $\CMod(C_0 X) \to \Des \mu$, and consequently the image of the pullback functor $\mu^*$ is the same as the image of the obvious forgetful functor $\Des\mu\to \CMod(C_0 Y)$.
\end{corollary}

\begin{proof}
By \cite[Theorem 5.6]{Crisp-descent} there is a unitary equivalence 
\[
\functor{F}:\CMod(C_0 X) \xrightarrow{\cong} \CComod( C_0 Y\otimes^{\h}_{C_0 X} C_0 Y)
\]
where $\CComod( C_0 Y\otimes^{\h}_{C_0 X} C_0 Y)$ is a $C^*$-category whose objects are splittings $\module{E}\to \mu^*\mu_* \module{E}$ as above; we refer the reader to \cite{Crisp-descent} for the precise definition. As in \cite{BenRou} we can use the isomorphisms \eqref{eq:BenRou} to turn such splittings into descent data, and vice versa, yielding a unitary equivalence 
\[
\functor{G}:\CComod( C_0 Y\otimes^{\h}_{C_0 X} C_0 Y) \xrightarrow{\cong} \Des\mu.
\]
It is then a simple matter to verify that the functor $\functor{D}$ is isomorphic to the composite $\functor{G} \functor{F}$.
\end{proof}

\begin{remarks}
\begin{enumerate}[\rm(1)]
\item The functor $\functor{D}:\CMod(C_0 X)\to \Des\mu$ can be an equivalence without the map $\mu$ satisfying the Beck-Chevalley condition; for example, if $X$ is a point then $\functor{D}$ is always an equivalence. (We owe this observation to an anonymous referee.) It may in fact be the case that the functor $\functor{D}$ is an equivalence for every proper surjection $\mu$; proving this in the absence of the Beck-Chevalley condition is beyond the scope of this note.
\item The finiteness condition on $\mu$ in Corollary \ref{cor:descent}  is quite restrictive, but interesting examples do exist. For instance, if $G$ is a complex semisimple Lie group then the Dixmier-Douady theory \cite{DD} combined with a theorem of Wallach \cite{Wallach-cyclic}  gives a Morita equivalence $C^*_r G\sim C_0(\widehat{H}/W)$ between the reduced group $C^*$-algebra of $G$ and the commutative $C^*$-algebra of Weyl-invariant continuous functions on the Pontryagin dual of a Cartan subgroup $H$ of  $G$. This is explained in \cite{Penington-Plymen}. Using this Morita equivalence, Corollary \ref{cor:descent} says that the functor of {tempered parabolic restriction} from $G$ to $H$ introduced in \cite{CCH-Compositio}  gives an equivalence between $\CMod(C^*_r G)$ and the category of descent data for the quotient mapping $\widehat{H}\to \widehat{H}/W$. 
\end{enumerate}
\end{remarks}

\subsection*{The Haagerup tensor product}

Turning now to the proof of Theorem \ref{thm:main}, we first establish some notation and recall some basic facts about the Haagerup tensor product, referring to \cite{BLM} for details. Let $B_1$ and $B_2$ be $C^*$-algebras, and let $A$ be a  $C^*$-algebra equipped with nondegenerate $*$-homomorphisms $A\to M(B_1)$ and $A\to M(B_2)$ into the multiplier algebras of $B_1$ and $B_2$. We use these homomorphisms to regard $B_1$ and $B_2$ as $A$-bimodules. For all positive integers $n$ and $m$ we consider the `external' matrix product 
\[
\begin{gathered}
\left(M_{n,m}\otimes B_1\right) \otimes \left( M_{m,n}\otimes B_2\right) \xrightarrow{D\otimes E\mapsto D\odot E} M_n\otimes (B_1 \otimes_{A} B_2) \\
(d\otimes b_1)\odot (e\otimes b_2) \coloneqq de\otimes (b_1\otimes b_2).
\end{gathered}
\]
Here $M_{n,m}$ means $n\times m$ complex matrices, $M_n$ means $M_{n,n}$, and an undecorated $\otimes$ means the algebraic tensor product over $\C$.

We equip each $M_n\otimes B$ with its $C^*$-algebra norm, and we equip $M_{n,m}\otimes B$ with the norm that it inherits as a subspace of $M_{\max(n,m)}\otimes B$. The {Haagerup seminorm} on $M_n\otimes (B_1 \otimes_{A} B_2)$ is defined by 
\[
\|F\|_{\h} \coloneqq \inf \left\{ \left. \|D\| \|E\| \ \right|\  m\in\N,\ D\in M_{n,m}\otimes B_1,\ E \in M_{m,n}\otimes B_2,\ F=D\odot E \right\}.
\]
The Haagerup tensor product $B_1\otimes^{\h}_A B_2$ is the separated completion of $B_1\otimes_A B_2$ in the $n=1$ Haagerup seminorm. The Haagerup norms extend to the matrix spaces $M_n\otimes(B_1\otimes_A^{\h} B_2)$, giving $B_1\otimes_A^{\h} B_2$ the structure of an operator space. 

The same operator-space structure can be defined by embedding $B_1\otimes_A B_2$ into the amalgamated free product $C^*$-algebra $B_1\ast_A B_2$ (cf. \cite[Theorem 3.1]{CES}, \cite[Lemma 1.14]{Pisier-Kirchberg},  \cite[p.515]{Ozawa}). The universal property of the free product then ensures that whenever $\rho_1:B_1\to C$ and $\rho_2:B_2\to C$ are $*$-homomorphisms satisfying $\rho_1(b_1 a)\rho_2(b_2) = \rho_1(b_1) \rho_2(ab_2)$ for all $b_1\in B_1$, $b_2\in B$, and $a\in A$, we obtain a completely contractive map
\begin{equation}\label{eq:rho1-rho2-def}
[\rho_1\rho_2] : B_1\otimes^{\h}_A B_2 \to C,\qquad b_1\otimes b_2 \mapsto \rho_1(b_1)\rho_2(b_2).
\end{equation} 

Returning to the setting of Theorem \ref{thm:main}, we let $\mu_1:Y_1\to X$ and $\mu_2:Y_2\to X$ be  continuous maps of locally compact Hausdorff spaces. Pullback of functions along $\mu_1$ and $\mu_2$ gives nondegenerate $*$-homomorphisms  $\mu_i^*:C_0 X\to M(C_0Y_i)$ which we use to form the Haagerup tensor product $C_0 Y_1 \otimes_{C_0 X}^{\h} C_0 Y_2$. Pullback along each coordinate projection $\pi_i:Y_1\times_X Y_2\to Y_i$ gives a nondegenerate $*$-homomorphism $\pi_i^*:C_0 Y_i \to M(C_0(Y_1\times_X Y_2))$, and since $\mu_1\pi_1=\mu_2\pi_2$ these homomorphisms induce, as in \eqref{eq:rho1-rho2-def}, a completely contractive map
\begin{equation*}\label{eq:pi1pi2}
\begin{aligned}
&[\pi_1^*\pi_2^*]: C_0 Y_1 \otimes_{C_0 X}^{\h} C_0 Y_2 \to M(C_0(Y_1\times_X Y_2)).
\end{aligned}
\end{equation*}
This  map $[\pi_1^*\pi_2^*]$ actually has image contained in $C_0(Y_1\times_X Y_2)$, and is precisely the map \eqref{eq:comparison} appearing in Theorem \ref{thm:main}.

We let $\flip$ denote the flip map $\tau: b_1\otimes b_2\mapsto b_2\otimes b_1$ from $C_0 Y_1\otimes_{C_0 X} C_0 Y_2$ to  $C_0 Y_2\otimes_{C_0 X} C_0 Y_1$.

\subsection*{Proof that (a) implies (b) in Theorem \ref{thm:main}}

If $[\pi_1^*\pi_2^*]$ is a completely bounded isomorphism, then taking operator-space adjoints (cf.~\cite[1.2.25]{BLM}) shows that the map $[\pi_2^*\pi_1^*]:C_0 Y_2\otimes^{\h}_{C_0 X} C_0 Y_1 \to C_0 (Y_2\times_X Y_1)$ is also a completely bounded isomorphism. We can factor the flip map $\tau$ as 
\[
C_0 Y_1 \otimes_{C_0 X}^{\h} C_0 Y_2 \xrightarrow{[\pi_1^*\pi_2^*]} C_0 (Y_1\times_X Y_2) \xrightarrow{\sigma^*} C_0 (Y_2\times_X Y_1) \xrightarrow{ [\pi_2^*\pi_1^*]^{-1} } C_0 Y_2\otimes_{C_0 X}^{\h} C_0 Y_1
\]
where $\sigma^*$ is the $*$-isomorphism induced by the flip $\sigma:Y_1\times_X Y_2\to Y_2\times_X Y_1$.  Isomorphisms of $C^*$-algebras are complete isometries, and $[\pi_1^*\pi_2^*]$ is a complete contraction, so we have
\begin{equation}\label{eq:tau-bound}
\pushQED{\qed}
\| \tau\|_{\cb} \leq \| [\pi_2^*\pi_1^*]^{-1}\|_{\cb}.\qedhere
\popQED
\end{equation}

\subsection*{Proof that (b) implies (c)  in Theorem \ref{thm:main}} 

The following computation is inspired by, but distinct from, Tomiyama's computation of the completely bounded norm of the transpose map \cite{Tomiyama}; cf.~Remark \ref{rem:transpose}. Recall that the completely bounded norm $\|\flip\|_{\cb}$ is the supremum over $n$ of the operator norms of the maps
\[
\id_{M_n}\otimes \flip : M_n \otimes (C_0 Y_1\otimes_{C_0 X}^{\h} C_0 Y_2) \to M_n \otimes (C_0 Y_2\otimes_{C_0 X}^{\h} C_0 Y_1).
\]

\begin{lemma}\label{lem:flip-lower-bound}
$\displaystyle\|\flip\|_{\cb} \geq \sqrt{\max_x \min_i \#\mu_i^{-1} x}$. In particular, if the right-hand side is infinite then so is $\|\flip\|_{\cb}$, and so (b) implies (c) in Theorem \ref{thm:main}.
\end{lemma}

\begin{proof}
Fix a point $x\in X$, and suppose that $\#\mu_i^{-1} x \geq m$ for $i=1,2$. We will prove that $\| \id_{M_m}\otimes \flip\|\geq \sqrt{m}$.

Let $y_1,\ldots, y_m \in Y_1$ and $z_1,\ldots, z_m\in Y_2$ be distinct points with $\mu_1(y_i)=\mu_2(z_i)=x$ for all $i$. Let $b_1,\ldots, b_m\in C_0 Y_1$  and $c_1,\ldots, c_m\in C_0 Y_2$ be a collection of functions with the following properties:
\[
\|b_i\|=\|c_i\|=b_i(y_i)=c_i(z_i)=1\text{ for all $i$; and }  b_ib_j=c_i c_j=0 \text{ when }i\neq j.
\]
Let $\omega\in \C$ be a primitive $m$th root of $1$, and consider the following element of $M_m\otimes (C_0 Y_1\otimes_{C_0 X}C_0 Y_2)$:
\[
F\coloneqq \sum_{i,j,k=1}^m \omega^{k(j-i)} e_{i,j}\otimes b_j\otimes c_k.
\]
Here $e_{i,j}$ denotes the matrix with $1$ in the $i,j$ position and zeros everywhere else.

We first claim that $\|F\|_{\h}\leq \sqrt{m}$. To show this we write $F=D\odot E$ where $D\in M_{1,m}\otimes M_{m,m} \otimes C_0 Y_1$ and $E\in M_{m,1}\otimes M_{m,m} \otimes C_0 Y_2$ are defined by 
\[
D \coloneqq \sum_{i,j,k=1}^m \omega^{-ki} e_{1,j}\otimes e_{i,k} \otimes b_j,\quad 
E \coloneqq \sum_{j,k=1}^m \omega^{kj} e_{j,1} \otimes e_{k,j} \otimes c_k.
\]

We have
\[
\| D\|^2 = \| DD^*\|   = \left\|  m \sum_{i,j=1}^m   e_{i,i} \otimes |b_j|^2 \right\| = m \left\| \sum_{j=1}^m |b_j|^2\right\| =m,
\]
while
\[
\|E\|^2 = \| E^* E\| = \left\| \sum_{j,k=1}^m e_{j,j}\otimes |c_k|^2 \right\| = \left\| \sum_{k=1}^m |c_k|^2\right\| =1,
\]
giving $ \|F\|_{\h} \leq \|D\| \| E\| \leq \sqrt{m}$.

Now let $\rho_1:C_0 Y_1 \to M_m$ and $\rho_2:C_0 Y_2\to M_m$ be the $*$-homomorphisms 
\[
\rho_1(b)\coloneqq \sum_{i=1}^m b(y_i) e_{i,i}\qquad \text{and}\qquad  \rho_2(c)\coloneqq U\left( \sum_{i=1}^m c(z_i)e_{i,i}\right)U^*,
\]
where $U \coloneqq m^{-1/2}\sum_{i,j=1}^m \omega^{ij}e_{i,j}$ (the unitary Fourier transform in $M_m$).
A straightforward computation shows that $(\id_{M_m}\otimes [\rho_2\rho_1] \flip)( F)  = \sum_{i,j} e_{i,j}\otimes e_{i,j}$, which is $m$ times an orthogonal projection. Since $[\rho_2\rho_1]$ is a complete contraction we find that
\[
\| (\id_{M_m}\otimes \flip) F\|_{\h} \geq \|(\id_{M_m}\otimes [\rho_2\rho_1] \flip)(F)\|_{M_m\otimes M_m} =m\geq \sqrt{m}\| F\|_{\h}.\qedhere
\]
\end{proof}

\subsection*{Proof that (c) implies (a) in Theorem \ref{thm:main}}

\begin{lemma}\label{lem:pi1pi2}
If the function $\displaystyle x\mapsto  \min_{i=1,2} \# \mu_i^{-1} x$ is uniformly bounded on $X$ then $[\pi_1^*\pi_2^*]$ is a completely bounded isomorphism. Explicitly, 
\[
\|[\pi_1^*\pi_2^*]\|_{\cb}\leq 1\quad \text{ and }\quad \displaystyle\|[\pi_1^*\pi_2^*]^{-1}\|_{\cb} \leq \sqrt{\max_{x\in X}\min_{i=1,2} \#\mu_i^{-1} x}.
\]
\end{lemma}

The proof relies on a couple of preliminary observations about the Haagerup norm. 

\begin{lemma}\label{lem:h-norm-sup-over-x}
For each $x\in X$ and $i=1,2$ let $\iota_{i,x} : \mu_i^{-1}x\to Y_i$ be the inclusion map, and consider the linear map
\[
\iota_{1,x}^*\otimes \iota_{2,x}^* : C_0 Y_1\otimes_{C_0 X} C_0 Y_2 \to C_0(\mu_1^{-1}x)\otimes C_0(\mu_2^{-1}x).
\]
For each  $n\geq 1$ and each $F\in M_n\otimes (C_0 Y_1\otimes_{C_0 X} C_0 Y_2)$ we have
\[
\|F\|_{\h} = \sup_{x\in X} \big\| (\id_{M_n}\otimes \iota_{1,x}^*\otimes \iota_{2,x}^*)F \big\|_{\h}.
\]
\end{lemma}

\begin{proof}
The Haagerup tensor product  $C_0Y_1\otimes_{C_0 X}^{\h} C_0 Y_2$ embeds completely isometrically in the amalgamated free product $C_0Y_1 \ast_{C_0 X} C_0 Y_2$, via the map $b_1\otimes b_2\mapsto b_1\ast b_2$. In every irreducible representation of this free product the algebra $C_0 X$ acts centrally, and thus by a character $a\mapsto a(x)$. So every irreducible representation of $C_0Y_1\ast_{C_0 X}C_0Y_2$ factors through some $\iota_{1,x}^*\ast \iota_{2,x}^*$. Taking the supremum over all of the irreducible representations thus gives the desired formula for $\|F\|_{\h}$. 
\end{proof}

The other ingredient in our proof of Lemma \ref{lem:pi1pi2} is the  following well-known fact:

\begin{lemma}\label{lem:h-min}
For all $n,m\in \N$, all $C^*$-algebras $A$, and all $F\in M_n\otimes \C^m\otimes A$, we have
\[
\|F\|_{\min} \leq \|F\|_{\h} \leq \sqrt{m} \|F\|_{\min}
\]
where $\h$ indicates the Haagerup norm on $M_n\otimes (\C^m\otimes A)$, and $\min$ indicates the $C^*$-algebra norm. The same bounds hold for $F\in M_n\otimes A\otimes \C^m$.
\end{lemma}

The bound $\|F\|_{\min}\leq \|F\|_{\h}$ holds for tensor products of arbitrary operator spaces  \cite[1.5.13]{BLM}.  Since we were not immediately able to find a reference for the other inequality, let us give a proof.

\begin{proof}
Take an arbitrary element
\[
F =  \sum_{i,j=1}^n \sum_{k=1}^m e_{i,j}\otimes \epsilon_k \otimes a_{i,j,k} \in M_n\otimes \C^m\otimes A
\]
where as before $e_{i,j}$ denotes the matrix with $1$ in the $i,j$ position and zeros everywhere else,  and  $\epsilon_1,\ldots,\epsilon_m\in \C^m$ are the standard basis elements. Factor $F$ as $D\odot E$, where 
\[
\begin{aligned}
D &\coloneqq \sum_{i=1}^n \sum_{k=1}^m e_{1,i}\otimes e_{i,k}\otimes \epsilon_k \in M_{1,n}\otimes M_{n,m}\otimes \C^m\qquad \text{and} \\
E &\coloneqq \sum_{i,j=1}^n \sum_{k=1}^m e_{i,1} \otimes e_{k,j} \otimes a_{i,j,k} \in M_{n,1}\otimes M_{m,n}\otimes A.  
\end{aligned}
\]
(Here $D\odot E$ is defined by using the Kronecker product to identify $M_{1,m}\otimes M_{m,m}\cong M_{m,m^2}$ and $M_{m,1}\otimes M_{m,m}\cong M_{m^2, m}$.)

We have $E^*E  = \sum_{k=1}^m F_k^* F_k$, where $F_k = \sum_{i,j=1}^n e_{i,j}\otimes a_{i,j,k}\in M_n\otimes A$. 
Since $DD^*$ is the identity in $M_n\otimes \C^m$, this computation gives the bound
\[
\|F\|^2_{\h} \leq  \|E\|^2 \leq m\sup_k\|F_k\|_{M_n\otimes A}^2 = m \|F\|_{\min}.
\]
This proves the result for $\C^m\otimes A$, and taking adjoints   gives the result for $A\otimes \C^m$.
\end{proof}

\begin{proof}[Proof of Lemma \ref{lem:pi1pi2}]
For each $x\in X$ and $i=1,2$ let $\pi_{i,x} : \mu_1^{-1}x\times \mu_2^{-1}x \to \mu_i^{-1}x$ be the projection onto the $i$th coordinate. Consider the commuting diagram
\begin{equation}\label{eq:pi1pi2-proof-diag}
\xymatrix@C=80pt@R=40pt{
C_0 Y_1 \otimes^{\h}_{C_0 X} C_0 Y_2 \ar[r]^-{[\pi_1^*\pi_2^*]} \ar[d]_-{\prod_{x} \iota_{1,x}^*\otimes\iota_{2,x}^*} & C_0(Y_1\times_X Y_2) \ar[d]^-{\prod_x (\iota_{1,x}\times \iota_{2,x})^*} \\
 \prod_{x\in X} C(\mu_1^{-1}x)\otimes^{\h}  C(\mu_2^{-1} x) \ar[r]^{\prod_x [\pi_{1,x}^* \pi_{2,x}^*]} & \prod_{x\in X} C_0(\mu_1^{-1}x \times \mu_2^{-1}x)
 }
 \end{equation}
 where $\prod$ denotes the $\ell^\infty$ product of operator spaces. Lemma \ref{lem:h-norm-sup-over-x} implies that the left-hand vertical arrow in \eqref{eq:pi1pi2-proof-diag} is a complete isometry. The right-hand vertical arrow is an injective $*$-homomorphism, hence it too is a complete isometry. If the quantity $\min_{i}\#\mu_i^{-1} x$ is uniformly bounded  then Lemma \ref{lem:h-min} implies that the bottom horizontal arrow is a completely bounded isomorphism of operator spaces: indeed, this map is a complete contraction, and its inverse has cb norm bounded by $\sqrt{\max_x \min_i \#\mu_i^{-1} x}$. Hence the commutativity of the diagram ensures that $[\pi_1^*\pi_2^*]$ is a completely bounded isomorphism onto a closed subspace of $C_0(Y_1\times_X Y_2)$, and an application of the Stone-Weierstrass theorem shows that this subspace is all of $C_0(Y_1\times_X Y_2)$. Since the vertical arrows in \eqref{eq:pi1pi2-proof-diag} are complete isometries, the previously noted bounds on the cb norms of the bottom horizontal arrow and its inverse imply the same bounds for $[\pi_1^*\pi_2^*]$ and its inverse. 
 \end{proof}

This completes the proof of Theorem \ref{thm:main}. The following formula for the cb norm of the flip $\flip$ falls out of the proof:

\begin{proposition}\label{prop:flip}
The flip map $\flip:C_0 Y_1\otimes^{\h}_{C_0 X}C_0 Y_2\to C_0 Y_2\otimes^{\h}_{C_0 X} C_0 Y_1$ associated to a pair of continuous maps $\mu_1:Y_1\to X$ and $\mu_2:Y_2\to X$ has
\[
\|\flip\|_{\cb} = \sqrt{\max_{x\in X} \min_{i=1,2} \# \mu_i^{-1} x}.
\]
(If one side of the equation is infinite, then so is the other.)
\end{proposition}

\begin{proof}
Theorem \ref{thm:main} shows that if one side is infinite then so is the other. When both sides are finite  Lemma \ref{lem:flip-lower-bound} gives one of the required inequalities, while combining Lemma \ref{lem:pi1pi2} and the bound \eqref{eq:tau-bound} gives the other.
\end{proof}

\begin{remark}  \label{rem:transpose}
Let $B$ be a $C^*$-algebra and consider the flip $\flip$ on $B\otimes^{\h} B$. Simple algebraic manipulations show that 
\[
\|\id_{M_n}\otimes \flip\| = \| T\otimes\id_{B\otimes^{\h}B}\|
\]
where $T$ is transposition on $M_n$. 
We have an isometric embedding $M_n\otimes(B\otimes^{\h}B)\into M_n\otimes (B\ast B)$, and it is interesting to compare the transpose maps on these two matrix spaces. Take $B=\C^m$. On the one hand, Proposition \ref{prop:flip} shows that the transpose on $M_n\otimes(B\otimes^{\h}B)$ has norm bounded by $\sqrt{m}$ (and equal to $\sqrt m$ when $n\geq m$). On the other hand, Tomiyama's results \cite{Tomiyama} show that if $m\geq 3$ then the transpose map on $M_n\otimes(B\ast B)$ has norm $n$ (because  $B\ast B$ surjects on to the $C^*$-algebra of $\operatorname{PSL}(2,\Z)$, and hence has irreducible representations of arbitrarily large dimension).
\end{remark}

\bibliographystyle{alpha}
\bibliography{descent}

\end{document}